\documentclass[a4paper]{amsart}
\usepackage{hyperref, enumitem}
\usepackage[dvipsnames]{xcolor} 
\usepackage{amsfonts,amsmath, amssymb,mathrsfs,mathtools,amsthm} 
\usepackage[doi=false, url=false, eprint=false, maxbibnames=10, firstinits=true]{biblatex}
\usepackage{comment}

\newtheorem{thm}{Theorem}[section]
\newtheorem{prop}[thm]{Proposition}
\newtheorem{lem}[thm]{Lemma}
\newtheorem{rem}[thm]{Remark}
\numberwithin{equation}{section}

\newcommand{\R}{\mathbb{R}}
\newcommand{\A}{\mathcal{A}}
\newcommand{\LRFH}{{\rm LRFH}}
\newcommand{\supp}{\operatorname{supp}}
\newcommand{\Crit}{\operatorname{Crit}}

\title{Two-boost problem\\ for the Newtonian potential at the infinity}
\author{Jagna Wi\'{s}niewska }
\address{Department of Mathematics, University of Augsburg}

\begin{document}
\maketitle

\begin{abstract}
The two-boost problem in space mission design asks whether two
points of position space can be connected by a Hamiltonian path on a fixed energy level set. We provide a positive answer for a class of systems having the same behaviour at infinity as the restricted three-body problem by relating it to the Lagrangian Rabinowitz Floer homology computed in \cite{Wisniewska2024}. The main technical challenge is to prove the boundedness of the corresponding sets of Reeb chords.
\end{abstract}

\section{Introduction}
The two-boost problem considers a very practical question in the space mission design:
Can we send a rocket in the gravitational field of one, or more, celestial bodies, using the engines only at the beginning and at the end of the journey?

The research on the two-boost problem has been initiated by the ground-breaking work of Hohmann on the attainability of heavenly bodies \cite{Hohmann}. Although this work was written almost a hundred years ago, the Hohmann transfer is still one of the crucial ingredients in space mission design \cite{Vallado}. Hohmann considered the two-boost problem in the setting of the Kepler problem, with the rocket moving under the influence of the gravitational field of just one heavenly body. The Hohmann transfer is a transfer between two circular orbits in the Kepler problem with the help of a Kepler ellipse which is tangent to the two circles. It requires two tangential boosts, one to transfer from the first circle to the ellipse and a second one to transfer from the ellipse to the second circle. This means that for the Kepler problem, two points in the
phase space can always be connected with the help of two boosts. The motivating
question for this paper is whether this continues to hold for the three body problem where the rocket is moving in the gravitational field of two celestial bodies instead of just one.

The general setup for the two-boost problem is as follows: Consider the cotangent bundle $T^*Q$ of a manifold $Q$ with its canonical exact symplectic form $\omega=d\lambda$, $\lambda=p\,dq$, and a Hamiltonian $H:T^*Q\to\R$. Given two points $q_0,q_1\in Q$ and an energy value $c$, the two-boost problem asks for the existence of a Hamiltonian orbit of energy $c$ connecting the cotangent fibers $T_{q_0}^*Q$ and $T_{q_1}^*Q$. In other words, we are interested in chords $(v,\eta)\in W^{1,2}([0,1], T^*Q)\times\R$ satisfying
\begin{equation}\label{Def(v,eta)}
\partial_tv=\eta X_H(v),\qquad v([0,1])\subseteq H^{-1}(c),\quad \textrm{and}\quad v(i)=q_i\quad \textrm{for}\quad i=0,1.
\end{equation}

These chords arise as critical points of the \emph{Rabinowitz action functional}
\begin{align*}
\A^{H-c}_{q_0,q_1} & : \mathscr{H}_{q_0,q_1}\times \R \to \R,\\
\A^{H-c}_{q_0,q_1}(v,\eta) & := \int_0^1 \lambda(\partial_t v)dt - \eta \int_0^1 (H-c)(v(t))dt
\end{align*}
defined on the path space
$$
\mathscr{H}_{q_0,q_1}:= \left\lbrace v \in W^{1,2}([0,1], T^*Q)\ \big|\ v(i)\in T^*_{q_i}Q\quad \textrm{for}\quad i=0,1\right\rbrace.
$$
This action functional was first defined by Rabinowitz in \cite{Rabinowitz1978} on the loop space to study the existence of periodic orbits on compact star-shaped hypersurfaces in $\R^{2n}$. In \cite{CieliebakFrauenfelder2009} Cieliebak and Frauenfelder defined a Floer-type homology of the Rabinowitz action functional on the loop space and called it the Rabinowitz Floer homology. In \cite{Merry2014} Merry
defined the Lagrangian Rabinowitz Floer homology $LRFH_*\left(\A^{H-c}_{q_0,q_1} \right)$ corresponding to the Rabinowitz action functional on the chord space. He showed that for high enough energy level sets in magnetic Hamiltonian systems the corresponding Lagrangian Rabinowitz Floer homology is well defined and non-trivial, hence providing a positive answer to the two-boost problem. The great advantage of using the Floer theory in the analysis of Hamiltonian systems is the continuation principle, which assures that a Floer-type homology is invariant under compact perturbations of the Hamiltonian. 
However, in both cases the Floer type homologies were defined under the assumption that the underlying manifold $Q$ is \emph{compact}.

In recent years several research groups in symplectic geometry independently developed new symplectic techniques to analyse Hamiltonian systems in \emph{non-compact} settings: Brugues \cite{Brugues2024}, Fontana \cite{Fontana2023}, Miranda \cite{Brugues2024, Fontana2023, Oms2021}, Peralta \cite{Fontana2023} and Oms \cite{Brugues2024, Fontana2023, Oms2021}\linebreak used the notion of $b$-symplectic and $b$-contact manifolds, which generalize the standard definitions to include forms with singularities, whereas in \cite{Ganatra2020} Ganatra, Pardon and Shende introduced the notion of Liouville sectors to describe a type of Liouville domains with boundary. However, the first generalizations of Floer-type homologies to the non-compact setting and the ones most relevant to this paper were the following: 

In \cite{Pasquotto2017} and \cite{Pasquotto2018} together with Pasquotto we extended the definition of Rabinowitz Floer homology to include examples of non-compact hypersurfaces. In \cite{Fauck2021} together with Fauck and Merry we computed the Rabinowitz Floer homology of a class of hyperboloids and used it to show the existence of periodic orbits on compact perturbations thereof. In our last paper \cite{Wisniewska2024} together with Cieliebak, Frauenfelder and Miranda we have extended the definition of the Lagrangian Rabinowitz Floer homology to the positive energy level sets of the following magnetic Hamiltonian:
\begin{equation}\label{DefH0}
\begin{aligned}
H_0(q_1,q_2,p_1,p_2) & :=\frac{1}{2}(p_1^2+p_2^2)+ p_1q_2-p_2q_1 = \frac{p_r^2}{2}+\frac{p_\theta^2}{2r^2}+p_\theta.
\end{aligned}
\end{equation} 
Here $(r,\theta)$ are the polar coordinates on the plane and $(p_r,p_\theta)$ are their conjugate momenta. More precisely, we have introduced the following set of smooth functions constant outside a compact set:
\begin{equation}\label{DefHset}
\mathcal{H}:=\left\lbrace h\in C^\infty(T^*\R^2)\ \Big|\ dh\in C_c^\infty(T^*\R^2),\quad h>0,\quad h-dh(p\partial_p)> 0\right\rbrace.
\end{equation}
and calculated the positive Lagrangian Rabinowitz Floer homology for Hamiltonians of the form $H_0-h$ with $h\in\mathcal{H}$ to be
$$
\LRFH_*^+\left(\A^{H_0-h}_{q_0,q_1}\right)=
\begin{cases}
\mathbb{Z}_2 & \textrm{for}\quad *=1/2,\\
0 & \textrm{otherwise}.
\end{cases}
$$
An immediate result of this theorem \cite[Thm. 1.2]{Wisniewska2024} is that $\Crit (\A^{H_0-h}_{q_0,q_1})\neq \emptyset$, which provides an affirmative answer to the two-boost problem in this setting. This was the first step towards answering the two-boost problem for the three body problem, and this paper, being a sequel to \cite{Wisniewska2024}, builds on this result.  

Returning to the question of two-boost problem in celestial mechanics, let us consider the \emph{planar circular restricted three-body problem}. It describes a Hamiltonian system of two bodies (the primaries) move under their mutual graviational attraction in circles around their common center of mass in a plane, a third body (ex. a satellite or a rocket) of negligible mass moves in the same plane in the gravitational field of the primaries. The evolution of this system in time can be described in the roating coordinate system as the Hamiltonian flow of the Hamiltonian $H:=H_0-V$, where $H_0$ is the magnetic Hamiltonian defined in \eqref{DefH0} corresponding to the sum of the kinetic energy and the Coriolis force, and $V$ is the gravitational potential of the primaries:
\begin{equation}\label{3BDPotential}
V(r,\theta)=\frac{1-\mu}{\sqrt{r^2+2r\mu \cos \theta +\mu^2}}+\frac{\mu}{\sqrt{r^2-2r(1-\mu) \cos \theta +(1-\mu)^2}}
\end{equation}

In contrast to the settings in \cite{CieliebakFrauenfelder2009} and \cite{Merry2014}, here the position space\linebreak $Q:=\R^2\setminus\{(-\mu, 0), (1-\mu,0)\}$ is non-compact, which reflects the possibility of collisions with the two primaries at $(-\mu, 0)$ and $(1-\mu,0)$ or an escape to infinity. Our goal is to answer the two-boost problem in this setting. Since the compactness due to collisions can be removed by the Birkhoff regularization \cite{Birkhoff1915}, which simultaneously regularizes collisions at both primaries, we will focus our attention on the noncompactness at the infinity in $\R^2$. For this we will consider smooth Hamiltonians on $T^*\R^2$, which coincide with the Hamiltonian of the restricted circular planar three-body problem outside of a compact set. 

We introduce a class of Hamiltonians defined
\begin{equation}\label{DefH}
H(q,p):= H_0(q,p) - V(q),
\end{equation}
where $H_0$ is the quadratic Hamiltonian defined in \eqref{DefH0} and $V\in C^\infty(\R^2)$ is a non-negative potential function, for which there exists $a>0$, such that for all $r>R_1$ we have:
\begin{equation}\label{DefV}
V(r, \theta) \leq \frac{a}{r}\qquad\textrm{and}\qquad \partial_r V(r,\theta) +\frac{2}{r}V(r,\theta)\geq 0.
\end{equation}
The main result of this paper is providing an affirmative answer to the two-boost problem for the class of Hamiltonians defined in \eqref{DefH}:

\begin{thm} \label{thm:ExistenceOfChords}
Consider a Hamiltonian $H$ as in \eqref{DefH} and fix $q_0,q_1 \in \R^2$, such that $|q_0|,|q_1|\leq R_1$. Denote by $B(R_1)$ a ball in $\R^2$ with radius $R_1$ and center at the origin.
If one of the following conditions is satisfied:
\begin{enumerate}[label=\arabic*.]
\item We choose energy value
\begin{equation}\label{CondC}
c> \max\left\lbrace \sqrt[3]{32a^2}, \sqrt{4a\left(3 R_1+2\sqrt[3]{2a}\right)}\right\rbrace,
\end{equation}
\item We choose the energy value $c>\sqrt{2aR_1}$ and, additionally, the potential function $V$ is rotationally invariant outside of $B(R_1)$, i.e. for all $r\geq R_1$ we have $\partial_\theta V(r,\theta)=0$,
\end{enumerate}
then there exists a solution $(v,\eta)\in W^{1,2}([0,1], T^*\R^2)\times\R$ to \eqref{Def(v,eta)}, such that\linebreak $v([0,1])\subseteq T^*B(R_1)$.
\end{thm}

The proof of Theorem \ref{thm:ExistenceOfChords} consists of two steps: the first step is the result from \cite[Thm. 1.2]{Wisniewska2024} mentioned above, where we consider Hamiltonians $h$ is a compact perturbation form the set $\mathcal{H}$ defined in \eqref{DefHset}, and we show that the corresponding positive Lagrangian Rabinowitz Floer homology $LRFH_*^+\left( \A_{q_0,q_1}^{H_0-h}\right)$ is non trivial. This immediately gives us that $\Crit^+ \left( \A_{q_0,q_1}^{H_0-h}\right)\neq \emptyset$.

The final step to prove Theorem \ref{thm:ExistenceOfChords} is the following proposition:
\begin{prop}\label{prop:chords}
Consider a Hamiltonian $H$ as in \eqref{DefH} and a pair $q_0,q_1 \in \R^2$, such that $|q_0|,|q_1|\leq R_1$. Let $\mathcal{H}$ be the set of Hamiltonians as in set \eqref{DefHset}. Then 
\begin{enumerate}[label=\arabic*.]
\item for every energy value $c>0$ satisfying \eqref{CondC}, there exists $h \in \mathcal{H}$, such that $\Crit^+ (\A^{H_0-h}_{q_0,q_1})\subseteq \Crit^+ (\A^{H-c}_{q_0,q_1})$.
\item for every value $c>\max\{\inf V, \sqrt{2aR_1}\}$, there exists $h \in \mathcal{H}$, such that $\Crit^+ (\A^{H_0-h}_{q_0,q_1})= \Crit^+ (\A^{H-c}_{q_0,q_1})$, provided that for all $r\geq R_1$ we have $\partial_\theta V(r,\theta)=0$.
\end{enumerate}
\end{prop}

This result together with \cite[Thm. 1.2]{Wisniewska2024} gives us existence of solutions to \eqref{Def(v,eta)} in the setting of Theorem \ref{thm:ExistenceOfChords} and concludes its proof.

The proof of this proposition relies on an observation that for high enough energies the Hamiltonian flow outside a compact set has hyperbolic behaviour. To get an intuition we can consider the potential where $V(r)=\frac{1}{r}$ for $r\geq R_1$. Then the Hamiltonian $H_0-V$ coincides with the rotating Kepler problem outside of a compact set. In the Kepler problem we know that the solutions are ellipses on the negative energy level sets, parabolas for the zero-level set, and hyperbolas on the positive level sets. In the rotating Kepler problem, the dynamics on a fixed level sets changes and on some level sets we can have all three types of trajectories: elliptic, parabolic and hyperbolic. However, for a fixed pair of endpoints we can choose the energy high enough, to assure that all the Floer trajectories passing through them have hyperbolic behaviour. In particular, once they leave a compact set, they cannot come back. For details see Subsection \ref{ssec:E&P}.

An important example of a Hamiltonian of the form as in \eqref{DefH} is a function which coincides with the Hamiltonian of the restricted planar circular three body problem outside of a compact set. In this setting we obtain the following result:
\begin{thm}\label{thm:Chords3BD}
For fixed constant $\mu \in (0, \frac{1}{2}]$ consider a Hamiltonian $H:= H_0-V$ with $H_0$ as in \eqref{DefH0} and the potential function $V$ is given by \eqref{3BDPotential} for $r\geq 1$. Then for any pair $q_0,q_1\in \R^2$ if we denote 
\begin{equation}\label{Cond3BD}
 R_1:=\max\left\lbrace 2(1-\mu), |q_0|, |q_1|\right\rbrace \qquad \textrm{and} \qquad a:= \frac{(1-\mu)R_1}{R_1-\mu}+\frac{\mu R_1}{R_1- (1-\mu)},
\end{equation}
and take an energy value $c>0$ satisfying \eqref{CondC} with the constants above,
there exists a solution $(v,\eta)\in W^{1,2}([0,1], T^*\R^2)\times\R$ to \eqref{Def(v,eta)}, such that $v([0,1])\subseteq T^*B(R_1)$.
\end{thm}
Naturally, this is not exactly the Hamiltonian corresponding to the restricted planar circular three body problem, as it is smooth on $T^*\R^2$ and does not take into account the singularities at the collisions, but as it is the same outside a compact set, it gives us an intuition on how the Hamiltonian flow behaves far away from the collision points. Our ultimate goal is to prove the two boost-problem for the regularized positive level sets of the restricted planar circular three body problem and we plan it to be the subject of our future work.


\section*{Acknowledgments}
I would like to thank prof. Urs Frauenfelder for all time devoted to mentoring me and all the fruitful discussions we had on the subject. I would also like to thank prof. Frauenfelder, prof. Cieliebak and all the other members of the symplectic geometry group for the warm and friendly welcome at the University of Augsburg.

My postdoctoral position at University of Augsburg is funded by the Deutsche Forschungsgemeinschaft (DFG, German Research Foundation) via the grant ``Himmelsmechanik, Hydrodynamik und Turing-Maschinen" - 541525489.

\section{Lemmas and proofs}
\subsection{Energy and angular momentum}\label{ssec:E&P}

In this subsection we will analyse the properties of Hamiltonians defined in \eqref{DefH}. More precisely, our aim is to show that for energy values high enough, the Hamiltonian flow on the energy hypersurface exhibits hyperbolic behaviour, i.e. the trajectories exiting the compact set, can not return back to it. One can understand it on the example of the Kepler problem: Recall, that in the setting of the inertial Kepler problem the Hamiltonian flow has three type of trajectories: elliptic for negative energy values, parabolic on the\linebreak $0$-energy level set and hyperbolic for positive energy. In the setting of the rotating Kepler problem the dynamics changes and on some energy levels there exist all three types of trajectories. The trajectory passing through a point $(r, \theta, p_r, p_\theta)$ is hyperbolic if and only if $H(r, \theta, p_r, p_\theta)-p_\theta>0$. However, for every $r>0$ there exists energy high enough such that for all $(r, \theta, p_r, p_\theta)\in H^{-1}(c)\cap T^*B(r)$ we have $p_\theta<c$. Consequently, all the trajectories passing through $H^{-1}(c)\cap T^*B(r)$ are hyperbolic: if they leave $H^{-1}(c)\cap T^*B(r)$, they will never enter it again.

In this subsection we generalize this analysis to the setting of Hamiltonians defined in \eqref{DefH}.
We will show that the if additionally the potential function is rotationally invariant, then for high energies the behaviour of the Hamiltonian flow is hyperbolic. As a consequence for high enough energies all the solutions to \eqref{Def(v,eta)} are confined to a compact set, as stated in the following proposition:

\begin{prop}\label{prop:RotBound}
Consider a Hamiltonian $H:= H_0-V$ with $H_0$ as in \eqref{DefH0} and the potential function $V$ as in \eqref{DefV}, with the additional assumption that for $r\geq R_1$, we have $\partial_\theta V(r,\theta)=0$. Then for every energy value $c>\sqrt{2aR_1^2}$ and every pair $q_0,q_1 \in \R^2$, such that $|q_0|,|q_1|\leq R_1$, all the solutions $(v,\eta)\in W^{1,2}([0,1], T^*\R^2)\times\R$ to \eqref{Def(v,eta)} satisfy $v([0,1])\subseteq T^*B(R_1)$.
\end{prop}
Before we prove the proposition, we will start with proving two lemmas, which apply to all Hamiltonians satisfying \eqref{DefH}, not necessarily rotationally invariant outside a compact set.

In the first lemma we will show that for all $r>R_1$ we can always choose an energy value $c$ high enough that for all $(r, \theta, p_r, p_\theta)\in H^{-1}(c)\cap T^*(B(r)\setminus B(R_1))$ we have $p_\theta<c$. 

\begin{lem}\label{lem:energy}
Let $H$ be the Hamiltonian as in \eqref{DefH} and choose an energy value $c> \sqrt{2aR_1^2} $. Then for all $(r,\theta, p_r, p_\theta)\in H^{-1}(c)$, such that $r\in [R_1, \frac{c^2}{2a})$, we have
$$
c-p_\theta>\frac{c^2-2ar}{2(c+r^2)}.
$$
\end{lem}
\begin{proof}
Denote $e:=H(r,\theta, p_r, p_\theta)-p_\theta$. Since $r\geq R_1$ then by \eqref{DefV}, we have
$$
e \geq \frac{p_\theta^2}{2 r^2}-V(r,\theta)\geq \frac{p_\theta^2}{2 r^2}-\frac{a}{r}= \frac{(c-e)^2}{2 r^2}-\frac{a}{r}=\frac{1}{2r^2}\left(e^2-2ce+c^2-2ar \right).
$$
The inequality above is satisfied if and only if
$$
e \in \left(c+r^2- \sqrt{(c+r^2)^2+2ar-c^2}, c+r^2+ \sqrt{(c+r^2)^2+2ar-c^2}\right).
$$
Thus if $r<\frac{c^2}{2a}$, then $e> 0$. We can calculate
$$
e > c+r^2- \sqrt{(c+r^2)^2+2ar-c^2} \geq \frac{c^2-2ar}{2(c+r^2)}.
$$

Note that for $r<\frac{c^2}{2a}$ we have
$$
\frac{d}{dr}\left(c+r^2- \sqrt{(c+r^2)^2+2ar-c^2}\right)=\frac{-2r\left(c+r^2+a-\sqrt{(c+r^2)^2+2ar-c^2} \right)}{\sqrt{(c+r^2)^2+2ar-c^2}}<0.
$$
Thus if $R_2\in \left(R_1,\frac{c^2}{2a}\right)$, then for all $(r,\theta, p_r, p_\theta)\in H^{-1}(c)$, such that $r\in [R_1, R_2]$, we have
$$
c-p_\theta\geq \frac{c^2-2aR_2}{2(c+R_2^2)}.
$$
\end{proof}

The following lemma assures that if $(v,\eta)$ is a solution to \eqref{Def(v,eta)} for Hamiltonians as in \eqref{DefH}, then the function $r\circ v$ cannot obtain its maximum in the set\linebreak $T^*(\R^2\setminus B(R_1))\cap \{H-p_\theta>0\}$ unless $v$ starts or ends there.

\begin{lem}\label{lem:empty}
For a Hamiltonian $H$ as in \eqref{DefH} we have:
$$
\{\{H,r\}=0\}\cap \{\{H,\{H,r\}\}\leq 0\}\cap \left\lbrace H-p_\theta > 0\right\rbrace\cap T^*(\R^2\setminus B(R_1))=\emptyset.
$$
\end{lem}

\begin{proof}
Let $x=(r,\theta, p_r, p_\theta) \in T^*(\R^2\setminus B(R_1))$. By \eqref{defXH1} we have
$$
\{H,r\}(x) = p_r, \qquad \textrm{and}\qquad \{H, \{H, r\}\}(x)=\frac{p_\theta^2}{r^3}+\partial_rV.
$$
Thus if $x \in T^*(\R^2\setminus B(R_1)\}\cap \{\{H,r\}=0\}$, then $p_r=0$.

Using \eqref{DefV} we can calculate
\begin{align*}
\{H, \{H, r\}\}(x) & =\frac{p_\theta^2}{r^3}+\partial_rV  = \frac{2}{r}\left( \frac{p_\theta^2}{2r^2}- V\right)+\left( \frac{2}{r}V+ \partial_rV\right)\\
& = \frac{2}{r}(H(x)-p_\theta)+\left( \frac{2}{r}V+ \partial_rV\right) \geq \frac{2}{r}(H(x)-p_\theta)>0.
\end{align*}
\end{proof}

Now we are ready to prove Proposition \ref{prop:RotBound} and show that for Hamiltonians as in \eqref{DefH} the corresponding solutions to \eqref{Def(v,eta)} are contained in a compact set, provided the Hamiltonian function is rotationally invariant and the energy is high enough. The main reason the following proof works for rotationally invariant potential functions, but not for an arbitrary potential function as in \eqref{DefV}, is that the flow of a rotationally invariant Hamiltonian preserves the angular momentum, whereas an arbitrary Hamiltonian of the form \eqref{DefH} may not.

\vspace*{.5cm}
\noindent \textit{Proof of Proposition \ref{prop:RotBound}:}

We will show all the solutions $(v,\eta)\in W^{1,2}([0,1], T^*Q)\times\R$ to \eqref{Def(v,eta)} satisfy $v([0,1])\subseteq T^*B(R_1)$, where $B(R_1)$ is a ball of radius $R_1$. 

We will argue by contradiction. Let $(v,\eta)\in W^{1,2}([0,1], T^*Q)\times\R$ to \eqref{Def(v,eta)} be a solution to \eqref{Def(v,eta)}. Denote $v(t)=(r(t), \theta(t), p_r(t), p_\theta(t))$ and suppose there exists $t_1 \in (0,1)$, such that $r(t_1)=\max_{t\in [0,1]}r(t)>R_1$. Then
\begin{equation}
\begin{aligned}
\frac{d}{dt} r(t)& =\eta \{H, r\}\circ v(t_1) = 0,\\ 
\frac{d^2}{dt^2}r(t) & =\eta^2\{H, \{H, r\}\}\circ v(t_1) \leq 0.
\end{aligned}
\end{equation}
In other words, 
$$
v(t_1)\in H^{-1}(c)\cap\{\{H,r\}=0\}\cap \{\{H,\{H,r\}\}\leq 0\}\cap T^*(\R^2\setminus B(R_1)).
$$

On the other hand, for $t_0:=\sup\{t \in [0, t_1]\ |\ r(t)\leq R_1\}$ we have $r(t_0)=R_1$. 
Since for all $t\in (t_0,t_1)$ we have $v(t)\in T^*(\R^2\setminus B(R_1))$, consequently, by our assumption on the potential $V$, we obtain 
$$
\frac{d}{dt}p_\theta(t)=\{H, p_\theta\}\circ v(t)= \partial_\theta H \circ v(t)= -\partial_\theta V \circ v(t)=0.
$$
By assumption $c>\sqrt{2aR_1}$ and Lemma \ref{lem:energy} we obtain
$$
H\circ v(t_1)-p_\theta(t_1)= H\circ v(t_0)-p_\theta(t_0) \geq \frac{c^2-2aR_1}{2(c+R_1^2)}>0.
$$
Consequently,
$$
v(t_1)\in \{\{H,r\}=0\}\cap \{\{H,\{H,r\}\}\leq 0\}\cap \left\lbrace H-p_\theta >0 \right\rbrace\cap T^*(\R^2\setminus B(R_1)).
$$
But from Lemma \ref{lem:empty}, we know that it is an empty set, which brings us a contradiction.

\hfill $\square$

\subsection{Compact perturbation}
The aim of this subsection is to prove Proposition \ref{prop:chords}.
We will first show how we construct the compactly supported perturbations from set $\mathcal{H}$, which agree with the potential function $V$ as in \eqref{DefV} on the set\linebreak $H^{-1}(c)\cap T^*B(R_1)$. 
Then, we show that for each chosen perturbation $h\in \mathcal{H}$, the solutions $(v,\eta)$ of \eqref{Def(v,eta)} of the corresponding Hamiltonian $H_0-h$ satisfy\linebreak $v([0,1])\subseteq H^{-1}(c)\cap T^*B(R_1)$.

Let $\chi:\R \to [0,1]$ be a smooth, not increasing function on $\R$, such that\linebreak $\inf \chi' \geq -2$ and 
\begin{equation}\label{DefChi}
\chi(x):=\begin{cases}
1 & \textrm{for}\quad x\leq 0,\\
0 & \textrm{for}\quad x\geq 1.
\end{cases}
\end{equation}

Let $H_0$ be the Hamiltonian defined in \eqref{DefH0} and $V$ the potential function as in \eqref{DefH}. For $c>0$ and $R_2>R_1$ we define the following smooth functions on $T^*\R^2$:
\begin{align}
\chi_0(r) & :=\chi \left(\frac{r-R_1}{R_2-R_1}\right), \label{DefChi0}\\
\chi_1(q,p) & := \chi(H_0(q,p)-\sup V-c),\label{DefChi1}\\
H_1 &:= H_0 - \chi_0\chi_1 V.\label{DefH1}
\end{align}

\begin{rem}\label{rem:relHvsH1}
Note that the three sets $H_1^{-1}(c), H^{-1}(c)$ and $H_0^{-1}(c)$ are subsets of $\chi_1^{-1}(1)$.
\end{rem}

Note that since $H_1^{-1}(c)\subseteq \chi_1^{-1}(1)$, the Hamiltonian vector field of $H_1$ on the set\linebreak $H_1^{-1}(c)\cap T^*(\R^2\setminus B(R_1))$ has the form:
\begin{equation}\label{defXH1}
X_{H_1}=p_r\partial_r+\left(\frac{p_\theta}{r^2}+1\right)\partial_\theta+\left(\frac{p_\theta^2}{r^3}+\chi_0\partial_rV+\chi_0' V\right)\partial_{p_r}+\chi_0 \partial_\theta V \partial_{p_\theta}.
\end{equation}

\begin{rem}\label{rem:equality}
For every $c>0$ and every $R_2>R_1$ the corresponding Hamiltonian $H_1$ defined in \eqref{DefH1} satisfies
\begin{align*}
T^*B(R_1)\cap H_1^{-1}(c)& =T^*B(R_1)\cap H^{-1}(c),\\
 H_1\big|_{T^*B(R_1)\cap H_1^{-1}(c)}& =H\big|_{T^*B(R_1)\cap H^{-1}(c)}.
\end{align*}
Consequently, for all $q_0,q_1\in B(R_1)$ we have
$$
\left\lbrace (v,\eta) \in \Crit \A_{q_0,q_1}^{H_1-c} \ \big|\ v([0,1])\subseteq T^*B(R_1) \right\rbrace \subseteq \Crit \A_{q_0,q_1}^{H-c}.
$$
\end{rem}
In the following Lemma we will show that Hamiltonians $H_1$ as defined in \eqref{DefH1} are of the form $H_0-h$ for some $h \in \mathcal{H}$, where $\mathcal{H}$ is the set defined in \eqref{DefHset}:
\begin{lem}
Consider functions $V$, $\chi_0$ and $\chi_1$ as in \eqref{DefH}, \eqref{DefChi0} and \eqref{DefChi1}, respectively. Then for every $c>0$ we have $\chi_0\chi_1 V+c
\in \mathcal{H}$, where $\mathcal{H}$ is the set of Hamiltonians defined in \eqref{DefHset}.
\end{lem}
\begin{proof}
First let us heck that $\chi_0\chi_1 V$ has compact support. By Remark \ref{rem:relHvsH1} we have
\begin{align*}
\supp \left( \chi_0\chi_1 V\right) & \subseteq \supp \chi_0 \cap \sup \chi_1= T^*B(R_2)\cap H_0^{-1}((-\infty, \sup V +c +1])\\
& = T^*B(R_2)\cap\left\lbrace \frac{1}{2}p_r^2+\frac{1}{2}\left(\frac{p_\theta}{r}+r\right)^2\leq \frac{1}{2}r^2+\sup V +c +1 \right\rbrace\\
& \subseteq \left\lbrace 
\begin{array}{c | c}
 & |p_r|\leq \sqrt{R_2^2+2(\sup V +c +1)}\\
\hspace*{-0.2cm}\smash{\raisebox{.5\normalbaselineskip}{$(r,\theta, p_r, p_\theta) \in T^*B(R_2)$}} & 
|p_\theta|\leq R_2 \left( R_2+ \sqrt{R_2^2+2(\sup V +c +1)}\right)
\end{array}\right\rbrace.
\end{align*}
Compactness of the set on the right-hand proves the first claim.

Further more we can calculate: 
\begin{align*}
d(\chi_0\chi_1 V)(p\partial_p) & = \chi_0 V d\chi_1(p\partial_p) = \chi_0 V \chi'(H_0-\sup V -c)dH_0(p\partial_p)\\
& = \chi_0 V \chi'(H_0-\sup V -c)\left(\frac{1}{2}|p_r|^2+\frac{p_\theta^2}{2r^2}+
H_0\right).\\
\supp \chi_1' & =\overline{\chi_1^{-1}((0,1))}=H_0^{-1}([\sup V +c, \sup V+c+1]),\\
d(\chi_0\chi_1 V)(p\partial_p) & \leq \chi_0 V \chi'(H_0-\sup V -c)\left(\frac{1}{2}|p_r|^2+\frac{p_\theta^2}{2r^2}+\sup V + c\right)\leq 0,
\end{align*}
where the last inequality uses the assumption that $V, \chi_0 \geq 0$ and $\chi'\leq 0$. Consequently, 
$$
\chi_0\chi_1 V+c-d(\chi_0\chi_1 V)(p\partial_p) \geq \chi_0\chi_1 V+c\geq c>0.
$$
\end{proof}

Now we will start analysing dynamical properties of Hamiltonians $H_1$ as in \eqref{DefH1}. More precisely, we will show that if the chords $(v,\eta)$ solving \eqref{Def(v,eta)} start\linebreak and end in $T^*B(R_1)$, then the function $r\circ v$ cannot obtain its maximum\linebreak in $H_1^{-1}(c)\cap T^*(B(R_2)\setminus B(R_1))= H_1^{-1}(c) \cap \chi_0^{-1}((0,1))$. Similarly, as in the setting of Lemma \ref{lem:empty}, to prove this property we will need to assume that $H_1-p_\theta$ is positive.

\begin{lem}\label{lem:emptySet}
Consider a potential function $V$ as in \eqref{DefV}. Choose two positive constants $c,e>0$ and set $R_2:=R_1+2\frac{a}{e}$. Then for the corresponding Hamiltonian $H_1$ defined as in \eqref{DefH1} we have
$$
\{\{H_1,r\}=0\}\cap \{\{H_1,\{H_1,r\}\}\leq 0\}\cap \left\lbrace H_1-p_\theta \geq e\right\rbrace\cap T^*(\R^2\setminus B(R_1))=\emptyset.
$$
\end{lem}
\begin{proof}
Let $x=(r,\theta, p_r, p_\theta) \in T^*(\R^2\setminus B(R_1))$. By \eqref{defXH1} we have
$$
\{H_1,r\}(x) = p_r, \qquad \textrm{and}\qquad \{H_1, \{H_1, r\}\}(x)=\frac{p_\theta^2}{r^3}+\chi_0\partial_rV+\chi_0' V.
$$
Thus if $x \in T^*(\R^2\setminus B(R_1)\}\cap \{\{H_1,r\}=0\}$, then $p_r=0$. 

For $x\in H_1^{-1}(c)\cap\{\{H_1,r\}=0\}\cap T^*(\R^2\setminus B(R_1)\}$, we can calculate
\begin{align}
\{H_1, \{H_1, r\}\}(x) & =\frac{p_\theta^2}{r^3}+\chi_0\partial_rV+\chi_0' V\nonumber \\
& = \frac{2}{r}\left( \frac{p_\theta^2}{2r^2}-\chi_0 V\right)+\chi_0\left( \frac{2}{r}V+ \partial_rV\right)+\chi_0' V\nonumber\\
& = \frac{2}{r}(H_1-p_\theta)+\chi_0\left( \frac{2}{r}V+ \partial_rV\right)+\chi_0' V\nonumber\\
& = \frac{2}{r}(H_1-p_\theta)+\chi_0\left( \frac{2}{r}V+ \partial_rV\right)-\frac{2}{R_2-R_1} V\nonumber\\
& \geq  \frac{2}{r}\left( H_1-p_\theta-\frac{a}{R_2-R_1}\right)\label{H1H1r}\\
& \geq \frac{2}{r}\left( e-\frac{a}{R_2-R_1}\right)=\frac{e}{r}>0,\nonumber
\end{align}
where the last two inequalities are a consequence of \eqref{DefV} and the assumption $R_2=R_1+2\frac{a}{e}$ respectively.

\end{proof}

To consider the setting of a potential function $V$ as in \eqref{DefV}, which is not rotationally invariant we need to keep in mind that the angular momentum $p_\theta$ is not preserved by the Hamiltonian flow any more. To remedy that, we introduce
more restrictions on the energy value of the Hamiltonian. 

The following lemma proves a property of Hamiltonian $H_1$, which assures that the corresponding solutions $(v,\eta)$ to \eqref{Def(v,eta)} starting and ending in $T^*B(R_1)$, do not obtain maximum of the function $r\circ v$ in the set $H_1^{-1}(c)\cap T^*(B(R_2)\setminus B(R_1))$:

\begin{lem}\label{lem:noMax}
Consider a potential function as in \eqref{DefV}. Choose an energy value $c> 0$ as in \eqref{CondC}.
Then for a Hamiltonian $H_1$ as in \eqref{DefH1} with $R_2:= \frac{1}{8a}\left( c^2+2aR_1\right)$ we have
$$
H_1^{-1}(c)\cap\{\{H_1,r\}=0\}\cap \{\{H_1,\{H_1,r\}\}\leq 0\}\cap T^*(B(R_2)\setminus B(R_1))=\emptyset.
$$
\end{lem}
\begin{proof}
By assumption \eqref{CondC} we have
\begin{align}
R_2 & = \frac{1}{8a}\left( c^2+2aR_1\right) \geq \frac{1}{8a}\left( 12aR_1+2aR_1\right)=\frac{7}{4} R_1 \nonumber\\
R_2 & = \frac{1}{8a}\left( c^2+2aR_1\right)\leq\frac{1}{8a}\left( c^2+\frac{1}{6}c^2\right)=\frac{7c^2}{48 a}<\frac{c^2}{2a}.\label{R2small}
\end{align}

Moreover, observe that for every $x=(r,\theta, p_r, p_\theta) \in T^*(\R^2\setminus B(R_1))\cap H_1^{-1}(c)$ we have
$$
c-p_\theta \geq \frac{p_\theta^2}{2 r^2}-\chi_0(r)V(r,\theta)\geq \frac{p_\theta^2}{2 r^2}-\frac{a}{r}.
$$
Consequently, by a similar argument as the one presented in the proof of Lemma \ref{lem:energy}, we can infer that for all $x\in H^{-1}(c)\cap T^*(B(R_2)\setminus B(R_1))$ we have
$$
c-p_\theta \geq \frac{c^2-2aR_2}{2(c+R_2^2)}.
$$
Now we can use relation \eqref{H1H1r}, the inequalities above, and assumption\linebreak $R_2=\frac{1}{8a}\left( c^2+2aR_1\right)$ to estimate
\begin{align*}
\{H_1, \{H_1, r\}\}(x) & \geq  \frac{2}{r}\left( c-p_\theta-\frac{a}{R_2-R_1}\right)\\
& \geq \frac{2}{r}\left(\frac{c^2-2aR_2}{2(c+R_2^2)}-\frac{a}{R_2-R_1}\right)\\
& = \frac{-4a R_2^2+R_2(c^2+2aR_1)-c(2a+cR_1)}{r(c+R_2^2)(R_2-R_1)}\\
& = \frac{\frac{1}{16a}\left( c^2+2aR_1\right)^2-c(2a+cR_1)}{r(c+R_2^2)(R_2-R_1)}\\
& = \frac{c^2\left(c^2-12aR_1\right)-32a^2c+4a^2R_1^2}{16ar(c+R_2^2)(R_2-R_1)}\\
& \geq \frac{8ac^2\sqrt[3]{2a}-32a^2c+4a^2R_1^2}{16ar(c+R_2^2)(R_2-R_1)}\\
& = \frac{2c\left(c\sqrt[3]{2a}-4a\right)+aR_1^2}{4r(c+R_2^2)(R_2-R_1)}\\
& \geq \frac{aR_1^2}{4r(c+R_2^2)(R_2-R_1)}>0.
\end{align*}
The estimates in the last three lines come from the conditions \eqref{CondC} on the energy value $c$.
\end{proof}

Now we are finally ready to prove Proposition \ref{prop:chords}:

\vspace*{.5cm}
\noindent \textit{Proof of Proposition \ref{prop:chords}:}
Observe that in view of Remark \ref{rem:equality}, to show that\linebreak $\Crit (\A^{H_1-c}_{q_0,q_1})\subseteq \Crit (\A^{H-c}_{q_0,q_1})$ it suffices to construct a Hamiltonian $H_1$, satisfying 
$$
\Crit \A_{q_0,q_1}^{H_1-c}=\left\lbrace (v,\eta) \in \Crit \A_{q_0,q_1}^{H_1-c} \ \big|\ v([0,1])\subseteq T^*B(R_1) \right\rbrace.
$$

To prove part:
\begin{enumerate}[label=\arabic*.]
\item We choose an energy value $c>0$ satisfying \eqref{CondC} define the Hamiltonian $H_1$ as in \eqref{DefH1} with $R_2:=\frac{1}{8a}(c^2+2aR_1)$ .
\item We choose an energy value $c \geq \max\{\inf V, \sqrt{2aR_1}\}$ and define the Hamiltonian $H_1$ as in \eqref{DefH1} with
\begin{equation}\label{parameters}
e:= \frac{c^2-2aR_1}{2(c+R_1^2)} \qquad \textrm{and} \qquad R_2:=R_1+2\frac{a}{e}.
\end{equation}
\end{enumerate}

For this choice of $c$ and $H_1$ will show that all $(v,\eta) \in \Crit (\A^{H_1-c}_{q_0,q_1})$, we have $v([0,1]) \subseteq T^*B(R_1)$, where $B(R_1)$ is a ball of radius $R_1$. 
We will argue by contradiction. Let $(v,\eta) \in \Crit (\A^{H_1}_{q_0,q_1})$, denote $v(t)=(r(t), \theta(t), p_r(t), p_\theta(t))$, and suppose there exists $t_1 \in (0,1)$, such that $r(t_1)=\max_{t\in [0,1]}r(t)>R_1$. Then
\begin{align*}
\frac{d}{dt} r(t)& =\eta \{H_1, r\}\circ v(t_1) = 0,\\ 
\frac{d^2}{dt^2}r(t) & =\eta^2\{H_1, \{H_1, r\}\}\circ v(t_1) \leq 0.
\end{align*}
In other words, we know that
\begin{equation}\label{v(t1)}
v(t_1)\in H_1^{-1}(c)\cap\{\{H_1,r\}=0\}\cap \{\{H_1,\{H_1,r\}\}\leq 0\}\cap T^*(\R^2\setminus B(R_1)).
\end{equation}

\noindent\textit{Proof of Part 1:}

By Lemma \ref{lem:noMax} we know that for energy values $c$ as in \eqref{CondC} we have
$$
H_1^{-1}(c)\cap\{\{H_1,r\}=0\}\cap \{\{H_1,\{H_1,r\}\}\leq 0\}\cap T^*(B(R_2)\setminus B(R_1))=\emptyset.
$$
Therefore, we know that
$$
v(t_1)\in H_1^{-1}(c)\cap\{\{H_1,r\}=0\}\cap \{\{H_1,\{H_1,r\}\}\leq 0\}\cap T^*(\R^2\setminus B(R_2)).
$$
By construction, we have $H_1|_{T^*(\R^2\setminus B(R_2))}=H_0$ and the Hamiltonian flow of $H_0$ preserves the angular momentum $p_\theta$. Consequently, 
$$
\textrm{if}\qquad t_0:= \max\{ t \leq t_1\ |\ r(t) \leq R_2\},\qquad\textrm{then}\qquad p_\theta(t_0)=p_\theta(t_1).
$$
Using \eqref{R2small} and Lemma \ref{lem:energy} we can estimate:
\begin{align*}
c-p_\theta(t_1)&=c-p_\theta(t_0)\geq \frac{c^2-2aR_2}{2(c+R_2^2)}\geq \frac{17c^2}{48(c+R_2^2)}>0,\\
\{H_1, \{H_1, r\}\}\circ v(t_1) & = \{H_0, \{H_0, r\}\}\circ v(t_1)=\frac{p_\theta^2(t_1)}{r^3(t_1)}=\frac{1}{r(t_1)}\left( c-p_\theta(t_1)\right)>0,
\end{align*}
which contradicts \eqref{v(t1)} and proves that for all $(v,\eta) \in \Crit (\A^{H_1-c}_{q_0,q_1})$, we have $v([0,1]) \subseteq T^*B(R_1)$.

\noindent\textit{Proof of Part 2:}

First observe that by Remark \ref{rem:equality} and Proposition \ref{prop:RotBound} for all pairs $q_0, q_1\in \R^2$, $|q_0|, |q_1|\leq R_1$ and all Hamiltonians $H_1$ of the form \eqref{DefH1} with energy values\linebreak $c> \sqrt{2aR_1}$ and $R_2>R_1$, 
we have $\Crit\A_{q_0,q_1}^{H-c}\subseteq \Crit\A_{q_0,q_1}^{H_1-c}$. Now we will prove that for Hamiltonian $H_1$ defined as in \eqref{DefH1} with parameter $R_2$ as in \eqref{parameters}, the opposite inclusion is satisfied.

By assumption the potential function $V$ is rotationally invariant, hence we also have $\partial_\theta(\chi_0\chi_1 V)=0$ outside of $B(R_1)$. Consequently, the flow of $H_1$ outside of $T*B(R_1)$ preserves the angular momentum:
$$
\frac{d}{dt}p_\theta(t)=\{H_1, p_\theta\}\circ v(t)= \partial_\theta H_1 \circ v(t)= -\partial_\theta (\chi_0 \chi_1 V) \circ v(t)=0.
$$
Consequently, 
$$
\textrm{if}\qquad t_0:= \max\{ t \leq t_1\ |\ r(t) \leq R_2\},\qquad\textrm{then}\qquad p_\theta(t_0)=p_\theta(t_1).
$$
and by Lemma \ref{lem:energy} we can estimate:
$$
c-p_\theta(t_1)=c-p_\theta(t_0)\geq \frac{c^2-2aR_1}{2(c+R_1^2)}=e>0.
$$
In other words, we obtain that
$$
v(t_1)\in \{\{H_1,r\}=0\}\cap \{\{H_1,\{H_1,r\}\}\leq 0\}\cap \left\lbrace H_1-p_\theta \geq e\right\rbrace\cap T^*(\R^2\setminus B(R_1)).
$$
But by our choice of $R_2$ in the definition of $H_1$ and by Lemma \ref{lem:emptySet} we know that the set above is empty, which brings us to a contradiction.

\hfill $\square$

\subsection{Three body problem} In this subsection we will apply the results of Theorem \ref{thm:ExistenceOfChords} to a Hamiltonian which outside a compact set coincides with the energy function in the restricted circular three body problem.

\noindent\textit{Proof of Theorem \ref{thm:Chords3BD}:}

To prove Theorem \ref{thm:Chords3BD}, it suffices to show that 
the potential function $V$ as in \eqref{3BDPotential} satisfies the conditions \eqref{DefV} for the constants $R_1$ and $a$ as in \eqref{Cond3BD}.

For $r\geq R_1>1-\mu>\mu>0$ we can calculate:
\begin{align*}
r V(r,\theta) & = \frac{(1-\mu)r}{\sqrt{r^2+2r\mu \cos \theta +\mu^2}}+\frac{\mu r}{\sqrt{r^2-2r(1-\mu) \cos \theta +(1-\mu)^2}}\\
& \leq \frac{(1-\mu)r}{r-\mu}+\frac{\mu r}{r- (1-\mu)} \leq \frac{(1-\mu)R_1}{R_1-\mu}+\frac{\mu R_1}{R_1- (1-\mu)},
\end{align*}
where the last inequality comes from the fact that
$$
\frac{d}{dr}\left( \frac{(1-\mu)r}{r-\mu}+\frac{\mu r}{r- (1-\mu)}\right) = -\mu(1-\mu)\left(\frac{1}{(r-\mu)^2} + \frac{1}{(r-(1-\mu))^2}\right)<0.
$$
This proves that for $r\geq R_1$ the potential function $V$ satisfies $V(r,\theta) \leq \frac{a}{r}$ with\linebreak $a:=\frac{(1-\mu)R_1}{R_1-\mu}+\frac{\mu R_1}{R_1- (1-\mu)}$.

To check the second condition of \eqref{DefV} we calculate
\begin{align*}
\partial_r V(r,\theta) & = -\frac{(1-\mu)(r+\mu \cos \theta)}{\left(r^2+2r\mu \cos \theta +\mu^2\right)^{\frac{3}{2}}}-\frac{\mu (r-(1-\mu)\cos \theta)}{\left(r^2-2r(1-\mu) \cos \theta +(1-\mu)^2\right)^{\frac{3}{2}}}\\
r \partial_r V+2V & = \frac{(1-\mu)}{\sqrt{r^2+2r\mu \cos \theta +\mu^2}}\left( 2 - \frac{r(r+\mu \cos \theta)}{r^2+2r\mu \cos \theta +\mu^2}\right)\\
& + \frac{\mu}{\sqrt{r^2-2r(1-\mu) \cos \theta +(1-\mu)^2}}\left( 2 - \frac{r(r-(1-\mu)\cos \theta)}{r^2-2r(1-\mu) \cos \theta +(1-\mu)^2}\right)\\
& = \frac{(1-\mu)\left(r^2+3r\mu \cos\theta + 2 \mu^2 \right)}{\left(r^2+2r\mu \cos \theta +\mu^2\right)^{\frac{3}{2}}}+\frac{\mu (r^2-3 r (1-\mu)\cos\theta + 2 (1-\mu)^2)}{\left(r^2-2r(1-\mu) \cos \theta +(1-\mu)^2\right)^{\frac{3}{2}}}\\
& \geq \frac{(1-\mu)\left(r^2-3r\mu + 2 \mu^2 \right)}{\left(r^2+2r\mu \cos \theta +\mu^2\right)^{\frac{3}{2}}}+\frac{\mu (r^2-3 r (1-\mu) + 2 (1-\mu)^2)}{\left(r^2-2r(1-\mu) \cos \theta +(1-\mu)^2\right)^{\frac{3}{2}}}\\
& = \frac{(1-\mu)(r-\mu)( r-2 \mu )}{\left(r^2+2r\mu \cos \theta +\mu^2\right)^{\frac{3}{2}}}+\frac{\mu (r-(1-\mu))(r - 2 (1-\mu))}{\left(r^2-2r(1-\mu) \cos \theta +(1-\mu)^2\right)^{\frac{3}{2}}}\geq 0,
\end{align*}
for all $r\geq 2(1-\mu)\geq 2\mu>0$. Thus both condition of \eqref{DefV} are satisfied and we can apply Theorem \ref{thm:ExistenceOfChords} to a Hamiltonian which which outside a compact set agrees with the energy function in the restricted circular three body problem.

\hfill $\square$

\printbibliography
\end{document}